\documentclass{amsart}

\usepackage{amsthm,amssymb,mathtools}
\usepackage[shortlabels]{enumitem}

\DeclareMathOperator{\ME}{\mathsf{E}}
\DeclareMathOperator{\prob}{\mathsf{P}}
\DeclareMathOperator{\var}{var}
\DeclareMathOperator{\cov}{cov}
\DeclareMathOperator{\corr}{corr}
\DeclareMathOperator{\indicatorfun}{\mathbf{1}}
\DeclareMathOperator{\hyper}{\sideset{_2}{_1}{\mathop F}}
\newcommand{\dto}{\xrightarrow{d}}
\newcommand{\deq}{\overset{d}{=}}
\newcommand{\cauchy}{\mathcal{C}(1)}

\newtheorem{thm}{Theorem}[section]
\newtheorem{prop}[thm]{Proposition}
\newtheorem{lemma}[thm]{Lemma}
\newtheorem{corollary}[thm]{Corollary}
\theoremstyle{definition}
\newtheorem{definition}[thm]{Definition}
\theoremstyle{remark}
\newtheorem{remark}[thm]{Remark}

\allowdisplaybreaks

\begin{document}

\title[Gaussian Volterra processes]{Gaussian Volterra processes: asymptotic growth and statistical estimation}

\author{Yuliya Mishura, Kostiantyn Ralchenko, and Sergiy Shklyar}
\address{Department of Probability Theory, Statistics and Actuarial Mathematics,
Taras Shevchenko National University of Kyiv,
64/13, Volodymyrs'ka St., 01601 Kyiv, Ukraine}
\email[Yu.\,Mishura]{yuliyamishura@knu.ua}
\email[K.\,Ralchenko]{kostiantynralchenko@knu.ua}
\email[S.\,Shklyar]{shklyar@univ.kiev.ua}

\thanks{The research by Yuliya Mishura has been supported by the Swedish Foundation for Strategic Research, grant no. UKR22–0017.
The research by Kostiantyn Ralchenko has been supported by the Sydney Mathematical Research Institute under the Ukrainian Visitors Program.}

\subjclass[2020]{60G22, 60G15, 60G17, 60G18, 62F12}

\date{}

\keywords{Gaussian Volterra process, asymptotic growth, long- and short-range dependence, parameter estimation, Ornstein--Uhlenbeck process}

\begin{abstract}
The paper is devoted to three-parametric self-similar Gaussian
Volterra processes that generalize fractional Brownian motion.
We study the asymptotic growth of such processes and the properties of long- and  short-range dependence.
Then we consider the problem of the drift parameter estimation for Ornstein--Uh\-len\-beck process driven by Gaussian Volterra process under consideration.
We construct a strongly consistent estimator and investigate its asymptotic properties. Namely, we prove that it has the Cauchy asymptotic distribution.
\end{abstract}

\maketitle

\section{Introduction}\label{s:1}
At the moment, the concept of fractional Brownian motion is
already so popular and mastered that it is not necessary to
explain what kind of stochastic process it is. However, we
still recall that fractional Brownian motion $B^H=\{B^H_t,
t\ge 0\}$ with Hurst index $H\in(0,1)$ is a zero mean
Gaussian stochastic process with covariance function
$$\ME B^H_tB^H_s = \tfrac12\left  (s^{2H}+t^{2H}-|s-t|^{2H}\right ).$$
It is neither Markovian nor semimartingale process except
the case $H=1/2$, when it is a standard Brownian motion. The
fact that it is non-Markovian makes it a very realistic
model in economics, technology, physics and biology, because
only the simplest processes are in reality Markovian. It
also has stationary increments and the property of self-similarity
with a coefficient equal to the Hurst index. On the
one hand, these are very convenient properties that allow
to construct a developed analytical apparatus and easier to
simulate trajectories. On the other hand, these properties
are quite restrictive and not always fulfilled, and
therefore have been criticized many times.

In order to
consider more general processes, there are several possibilities how to move. As one of possible generalizations, the multifractional
Brownian motion was introduced, see e.g.~\cite{Ayache2000,R2011}.
This process has a Hurst function
$H_t$ instead of fixed Hurst index. Such an object is quite
flexible, but measuring and modeling the Hurst function is
not always easy because it's usually easier to evaluate some
numerical parameter than a function that changes over time.
However, as we have already said, the presence of a single
index strongly regulates the properties of the
process. Another generalization was studied in the paper \cite{sovi}, where  the general Fredholm and Volterra Gaussian processes of the form
\[
X_t=\int_0^tK(t,s)dW_s \quad\text{and}\quad X_t=\int_0^TK_T(t,s)dW_s,
\quad t\in[0,T],
\]
where $W$ is a Wiener process, were introduced. Such approach is very interesting, but   the less facts we know about the kernel  of such a representation, the less useful properties of the corresponding process can be established. In this connection, in the paper \cite{mishshevshk}  we considered a kernel of the form
\[
K(t,s) = a(s) \int_s^t b(u)\,c(u-s) \, du,
\]
with   functions $a$, $b$, $c$  that guarantied the existence of such integral. This is a more concrete representation, but the properties of such a process, being very interesting in their specific features, are also   limited by its existence and the smoothness of trajectories.
Therefore, in this work, as well as in papers
\cite{Part1,Part2} we decided to go the following way: to take the
integral representation of the fractional Brownian motion (see \cite{Norros1999})
$$B^H_t=C_H \int_0^t K_H(t,s) \, dW_s, \quad t\ge 0,$$
where $C_H$ is some constant normalizing factor, and for $H\in(\frac12,1)$,
\begin{equation}\label{eq:KH-fBm}
K_H(t,s)=s^{1/2-H} \int_s^t u^{H-1/2} (u-s)^{H-3/2} \, du,
\end{equation}
and replace in it the unique index $H$ by three possibly
different indices $\alpha, \beta, \gamma$, getting a kernel
of the form
$$K(t,s)=s^{\alpha} \int_s^t u^{\beta} (u-s)^{\gamma} \, du,$$
and consider a process of the form
\begin{equation}\label{mainint}
X_t=\int_0^ts^{\alpha}\! \int_s^t u^{\beta} (u-s)^{\gamma}\,du\,dW_s,\,t\ge 0.
\end{equation}
Of course, $\alpha, \beta, \gamma$ should be chosen in a way
supplying the existence and at least continuity of the
integral of this kernel w.r.t.\@{} a Wiener process.
With this kernel, we get a self-similar process
\cite[Proposition~1]{Part1}, but, generally speaking, without
restrictive property of stationarity of increments
\cite[Theorem~2]{Part1}.
Moreover, the choice of indices allows us to get the process
with both short and long memory while in the fractional
Brownian case the representation~\eqref{eq:KH-fBm}
for the kernel $K_H(t,s)$ exists only for $H>1/2$,
and the respective fBm can have a long memory only. What is
important, we can integrate the nonrandom functions w.r.t.\ this process in the unified way
\cite[Definition~1]{Part2},
while the integration w.r.t.\ the fBm differs essentially for $H$
depending if it is more or less than $1/2$
\cite[Section~2.2]{Norros1999}.
In particular, the class of functions integrable w.r.t.\ the fBm depends on $H$.
So, with the help of the unified kernel
$K(t,s)$ we can model a variety of properties, that is
impossible for fBm.

The paper is organized as follows.
In Section~\ref{s:2} we investigate the asymptotic growth of the process $X$ and prove almost sure upper bounds for it.
In Section~\ref{s:3} we study the asymptotic behavior of the correlations between $X_1$ and the increment $X_{t+1}-X_t$ and establish the conditions for short- and long-range dependence for the process \eqref{mainint} in terms of $\alpha$, $\beta$ and $\gamma$.
Section~\ref{s:4} is devoted to Ornstein--Uhlenbeck model, where the noise is a Gaussian Volterra process $X$ of the form \eqref{mainint}. For such model we construct a strongly consistent drift parameter estimator and study its asymptotic properties. Namely, we prove its convergence in law to a certain Cauchy-type distribution.
Appendix~\ref{s:A} contains auxiliary results concerning self-similar processes. In Appendix~\ref{s:B} we establish asymptotic properties of some deterministic integrals arising in Section~\ref{s:4}.

\section{Asymptotic growth of the process $X$}\label{s:2}
In what follows, let $\{X_t,t>0\}$ be a
self-similar process with self-similarity
index $\rho$. Based on this process, and applying approach from
\cite{Lamperti1962,Yazigi2015}, we construct another process,
$\{Y_s,s\in\mathbb{R}\}$,
$Y_s = e^{-\rho s} X_{e^s}$,
which is stationary
(see Proposition~\ref{prop:sstostat-part1} in Appendix~\ref{apx:apx}).
Then we use the known results on the asymptotic growth of the stationary
process, $Y$,
to obtain bounds for the asymptotic growth of the process~$X$.

\begin{thm}\label{thm:ssGrow2}
Let
$\{X_t, t>0\}$
be a path-continuous
self-similar zero-mean Gaussian process
with self-similarity exponent $\rho\in\mathbb{R}$ and
with $\ME X_1^2 = \sigma^2$, $\sigma>0$.
Then
\begin{gather}
\limsup_{t\to+\infty}
\frac{|X_t|}{t^\rho \, \sqrt{2 \log \log t}} \le \sigma
\quad \mbox{a.s.},
\nonumber\\
\text{and }\;\sup_{t\ge 3}
\frac{|X_t|}{t^\rho \, \sqrt{\log \log t}}
< \infty
\quad \mbox{a.s.}
\label{neq:ssGrow2-p2}
\end{gather}
\end{thm}
\begin{proof}
As it was mentioned above, according to Proposition~\ref{prop:sstostat-part1},
 $\{Y_s,\; s\in\mathbb{R}\}$
with
\[
Y_s = e^{-\rho s} \, X_{e^s}
\]
is a continuous zero-mean stationary Gaussian process.
Notice  that
\[ \ME Y_s^2 = e^{-2\rho s} \ME X_{e^s}^2 = \ME X_1^2 = \sigma^2 .\]
Thus,
$Y_s / \sigma \sim \mathcal{N}(0,1)$ for all $s$.

By \cite[Theorem~1.4]{Marcus1972}
(see Theorem~\ref{thm:Marcus-theorem} in Appendix~\ref{apx:apx}),
\[
\limsup_{s\to+\infty} \frac{|Y_s|}{\sigma\, \sqrt{2 \log s}}
\le 1.
\]
Thus,
\[
\limsup_{t\to+\infty}
\frac{|X_t|}{t^\rho \, \sqrt{2 \log\log t}} =
\limsup_{s\to+\infty} \frac{|X_{e^s}|}{ e^{\rho s} \, \sqrt{2 \log s}} =
\limsup_{s\to+\infty} \frac{|Y_s|}{\sqrt{2 \log s}}
\le \sigma.
\]
As the process $\{t^{-\rho} \, |X_t| \mathbin{/} \sqrt{2 \log \log t},
\; \allowbreak t \ge 3\}$
is bounded on any finite interval, this implies
\begin{equation*}
\sup_{t\ge 3}
\frac{|X_t|}{t^{\rho} \, \sqrt{2 \log\log t}} < \infty,
\end{equation*}
whence \eqref{neq:ssGrow2-p2} follows.
\end{proof}

\begin{remark}
1) We did not specify  in Theorem~\ref{thm:ssGrow2}
whether the self-similarity exponent $\rho$ is negative, zero or positive.
However, a non-trivial self-similar process with negative
self-similarity exponent cannot be continuous at point $0$
due to Theorem~\ref{thm:Hnonneg}.
In this connection,  we assumed in Theorem~\ref{thm:ssGrow2}  that
the process $X$ is continuous on $(0,+\infty)$ rather than
on $[0,+\infty)$.

2) Of course, there is a general characterization of self-similar Volterra process $X_t=\int_0^t K(t,s) dW_s$  in terms of the kernel $K$: the process $X$ is $\rho$-self-similar if and only if for any $a>0$, $t,s>0$
\[
\int_0^{t\wedge s}K(at,az)K(as,az)dz=a^{2\rho-1}\int_0^{t\wedge s}K(t,z)K(s,z)dz.\]
\end{remark}
Now, let us proceed with the process $X$ from \eqref{mainint} and apply Theorem \ref{thm:ssGrow2}.
 \begin{thm}\label{thm:2.3}
Let
\begin{equation}\label{neq:procXconds}
\alpha>-\frac12,\qquad
\gamma>-1, \qquad
\alpha+\beta+\gamma > -\frac32 ,
\end{equation}
and let $\{X_t,\; t\ge 0\}$ be a path-continuous
modification of the process \eqref{mainint}, existing for these values of parameters, according to Theorems 1 and 6 from \cite{Part1}.

Then
\begin{gather}
\limsup_{t\to\infty}
\frac{|X_t|}{t^{\alpha+\beta+\gamma+1.5} (2 \log \log t)^{1/2}}
\le \ME X_1^2 < \infty,
\nonumber
\\
\label{neq:thm-ssGrow1-asj-a}
\sup_{t\ge 3}
\frac{|X_t|}{t^{\alpha+\beta+\gamma+1.5} (\log \log t)^{1/2}}
< \infty
\quad \mbox{a.s.}
\end{gather}
\end{thm}

\begin{proof}
The process $X$ is Gaussian; it has zero means and continuous paths.
By \cite[Proposition~1]{Part1}, the process $X$
is self-similar with self-similarity exponent
\[
H = \alpha + \beta + \gamma + \frac32 .
\]
By Theorem~\ref{thm:ssGrow2}, the process $X$ satisfies
the desired inequalities a.s.
\end{proof}

\section{Long- and  short-range dependence}\label{s:3}
In this section, we define the long-memory or short-memory
properties of a stochastic process.
\begin{definition}\label{def:shlong-dep-i1}
	Let $\{X_t,\; t\ge0\}$ be a stochastic process
	with finite second moments (i.e.\@{} $\ME X_t^2 < \infty$
	for all $t\ge 0$).

	If $\sum_{n=1}^\infty |\cov(X_1,\, X_{n+1}-X_n)|<\infty$,
	then the process $X$ is called
	\textit{short-range dependent}.

	Otherwise, if $\sum_{n=1}^\infty |\cov(X_1,\, X_{n+1}-X_n)|=\infty$,
	then the process $X$ is called
	\textit{long-range dependent}.
\end{definition}

\begin{remark}
In some cases, it is not natural to connect long- or short-range dependence to the value $X_1$. However, for self-similar process $X_t,\, t\ge 0$ with index $\rho$ for any fixed $k>0$ and $n\ge k$ 
\[
\cov(X_k, X_{n+1}-X_n)=k^{2\rho}\cov\left (X_1, X_{\frac{n+1}{k}}-X_{\frac{n}{k}}\right ),
\]
therefore asymptotics will be the same no matter where you start.
\end{remark}

If $X_0=0$, $\var(X_1)>0$ and the process $X$ has
positively correlated increments,
then the definition of short-range and long-range dependence
can be simplified. The process $X$ is short-range dependent
if $\lim_{t\to +\infty} \cov(X_1, X_t)$ is finite,
and $X$ is long-range dependent if
$\lim_{t\to +\infty} \cov(X_1, X_t) = +\infty$.
Similar equivalence holds true for the process with
negatively correlated increments:
$X$ is short-range dependent
if $\lim_{t\to +\infty} \cov(X_1, X_t)$ is finite,
and $X$ is long-range dependent if
$\lim_{t\to +\infty} \cov(X_1, X_t) = -\infty$.

Unless $\cov(X_1,\, X_{n+1} - X_n) = 0$ for all $n$ great enough,
we can make the process $X$ long-range or short-range dependent
by multiplying it by a non-random function.
In order to reduce the effect of ``variable scale'',
we consider Pearson correlation instead of covariance,
and study the convergence of the series
$\sum_{n=1}^\infty |\corr(X_1,\, X_{n+1}-X_n)|$.

For the rest of the section, we consider the stochastic
process $\{X_t, t\ge 0\}$ defined
in \eqref{mainint}.
Due to \cite[Eq.~(8)]{Part1},
\begin{equation}\label{eq:invrcov}
\begin{split}
\MoveEqLeft
	\ME[(X_{t_2}-X_{t_1}) (X_{t_4}-X_{t_3})]
	\\
	&= 	\int_0^{t_2\wedge t_4} s^{2\alpha}
	\int_{s\vee t_1}^{t_2} u^\beta (u{-}s)^\gamma \, du
	\int_{s\vee t_3}^{t_4} v^\beta (v{-}s)^\gamma \, dv \, ds > 0 .
\end{split}
\end{equation}
if $0\le t_1<t_2$ and $0\le t_3 < t_4$.
Thus, the increments of the process $X$ are positively correlated.

In order to find whether the process $X$ is long-range or short-range
dependent,
we evaluate the asymptotics of
\[
	\cov(X_1,\, (X_{t+1}-X_t))
	\quad \mbox{and} \quad
	\corr(X_1,\, (X_{t+1}-X_t))
\]
and check the convergence of the series
\[
	\sum_{n=1}^\infty \cov(X_1,\, (X_{n+1}-X_n))
	\quad \mbox{and} \quad
	\sum_{n=1}^\infty \corr(X_1,\, (X_{n+1}-X_n))
\]
(all terms of the series are positive, so we omit an ``absolute value'' sign).

\begin{thm}
 \label{thm:mem1}
	Let assumptions  \eqref{neq:procXconds} hold true,
	and let the process $X$ be defined in \eqref{mainint}.
	Then
\begin{equation}\label{eq:mem1-asy}
	\ME [X_1 \, (X_{t+1} - X_t)] \sim
	\frac{\mathrm{B}(2\alpha{+}1, \, \gamma{+}1)}
	{2\alpha{+}\beta{+}\gamma{+}2} \,
	t^{\beta+\gamma}
	\quad \mbox{as} \quad t\to{+}\infty .
\end{equation}
	If $\beta+\gamma <-1$, then
\[
	\sum_{n=1}^\infty \cov(X_1,\, (X_{n+1}-X_n)) < \infty
\quad \mbox{and} \quad
	\int_0^\infty \cov(X_1,\, (X_{t+1}-X_t)) \, dt
	< \infty,
\]
and thus the process $X$ is short-range dependent.
Otherwise, if $\beta+\gamma\ge -1$, then
\[
	\sum_{n=1}^\infty \cov(X_1,\, (X_{n+1}-X_n)) = \infty
\quad \mbox{and} \quad
	\int_1^\infty \cov(X_1,\, (X_{t+1}-X_t)) \, dt
	= \infty,
\]
and the process $X$ is long-range dependent.
\end{thm}

\begin{proof}
Due to \eqref{eq:invrcov}, for all $t\ge 1$
\begin{equation}\label{eq:invrcov-1m}
	\ME[X_1 (X_{t+1}-X_{t})] =
	\int_0^1 s^{2\alpha}
	\biggl(\int_{s}^{1} u^\beta (u{-}s)^\gamma \, du\biggr)
	\biggl(\int_{t}^{t+1} v^\beta (v{-}s)^\gamma \, dv \biggr) \, ds.
\end{equation}
Obviously, we have uniform convergence
\[
	\frac{(t+x)^\beta (t+x-s)^\gamma}{t^{\beta+\gamma}}
	= \left(1+\frac{x}{t}\right)^\beta
	  \left(1+\frac{x{-}s}{t}\right)^{\!\gamma} \to 1
	\quad \mbox{as} \quad
	t\to{+}\infty,\quad
	x,s\in[0,1].
\]
Then
\begin{equation}\label{eq:temp316}
	t^{-\beta-\gamma} \int_t^{t+1} v^\beta (v{-}s)^\gamma \, dv
	= \int_0^1
	\frac{(t+x)^\beta (t+x-s)^\gamma}{t^{\beta+\gamma}} \, dx
	\to 1
\end{equation}
uniformly in $s\in[0,1]$ as $t\to+\infty$.
Due to \eqref{eq:invrcov-1m} and \eqref{eq:temp316},
\begin{align*}
	\MoveEqLeft t^{-\beta-\gamma} \ME[X_1 \, (X_{t+1}-X_t)]
\\
	&=
	\int_0^1 s^{2\alpha}
	\biggl(\int_{s}^{1} u^\beta (u{-}s)^\gamma \, du\biggr)
	\biggl(t^{-\beta-\gamma} \int_{t}^{t+1} v^\beta (v{-}s)^\gamma
	\, dv \biggr) \, ds
	\\
	&\to
	\int_0^1 s^{2\alpha}
	\biggl(\int_{s}^{1} u^\beta (u{-}s)^\gamma \, du\biggr)
	\, ds
	=
	\int_0^1 u^{\beta} \int_0^u s^{2\alpha} (u-s)^\gamma \, ds \, du
	\\
	&=
	\int_0^1 u^{\beta + 2\alpha + \gamma + 1}
	\mathrm{B}(2\alpha{+}1,\, \gamma{+}1) \, du
	= \frac{\mathrm{B}(2\alpha{+}1,\, \gamma{+}1)}{2\alpha+\beta+\gamma+2}
\end{align*}
as $t\to +\infty$, which is equivalent to \eqref{eq:mem1-asy}.

The convergence of the integrals and series
follows from well-known criteria.
We use the fact that, since the process $\{X_t,\; t\ge 0\}$
is mean-square continuous,
the covariance
$\cov(X_1,\, (X_{t+1}-X_t)) = \ME [X_1\, (X_{t+1}-X_t)]$
is continuous in $t$ on $[0,+\infty)$.
\end{proof}

\begin{lemma}\label{lem:incpcovas}
	Let conditions \eqref{neq:procXconds} hold true,
	and let the process $X$ be defined in \eqref{mainint}.
The asymptotics of the incremental variance $\ME(X_{t+1}-X_t)^2$
as $t\to+\infty$ is the following:
\begin{eqnarray*}
	&	\ME(X_{t+1}-X_t)^2 \sim
	\frac{\mathrm{B}(\gamma{+}1, {-}2\gamma{-}1) \, t^{2\alpha+2\beta}}
	{(\gamma+1) (2\gamma+3)}
	& \mbox{if $\gamma<{-}\frac12$}, \\
	& \ME(X_{t+1}-X_t)^2 \sim
	t^{2\alpha+2\beta} \log t
	& \mbox{if $\gamma={-}\frac12$}, \\
	& \ME(X_{t+1}-X_t)^2 \sim
	\mathrm{B}(2\alpha{+}1, 2\gamma{+}1)\,
	t^{2\alpha+2\beta+2\gamma+1}
	& \mbox{if $\gamma>{-}\frac12$}.
\end{eqnarray*}
\end{lemma}
\begin{proof}
The process $X$ is self-similar
with exponent $\alpha+\beta+\gamma+\frac32$
according to \cite[Proposition~1]{Part1}.
Hence
\[
	\ME (X_{t+1}-X_t)^2 =
	t^{2\alpha+2\beta+2\gamma-3} \ME (X_{1+1/t} - X_1)^2.
\]
We apply \cite[Proposition~2]{Part1} for the asymptotics of
$\ME (X_{1+1/t} - X_1)^2$, and we immediately obtain
the conclusion of Lemma~\ref{lem:incpcovas}.
\end{proof}

Now we consider the asymptotic behaviour of the correlation
\[
	\corr(X_1,\, (X_{t+1}{-}X_t)) =
	\frac{\ME [X_1\, (X_{t+1}{-}X_t)]}
	{\sqrt{\ME X_1^2} \sqrt{\ME (X_{t+1} {-} X_t)^2}}.
\]
\begin{thm}
 \label{thm:asypcorr}
	Let conditions \eqref{neq:procXconds} hold true,
	and let the process $X$ be defined in \eqref{mainint}.
Then the asymptotics of the correlation $\corr(X_1,\, (X_{t+1}-X_t))$ is
\begin{eqnarray}
	& \corr(X_1,\, (X_{t+1}{-}X_t)) \sim
	c \, t^{\gamma - \alpha}
	& \mbox{if $\gamma<-\frac12$},
	\label{eq:prasypcorr-1}
	\\
	& \corr(X_1,\, (X_{t+1}{-}X_t)) \sim
	c \, t^{-0.5 - \alpha} (\log t)^{-1/2}
	& \mbox{if $\gamma=-\frac12$},
	\label{eq:prasypcorr-2}
	\\
	& \corr(X_1,\, (X_{t+1}{-}X_t)) \sim
	c \, t^{-0.5 - \alpha}
	& \mbox{if $\gamma>-\frac12$},
	\label{eq:prasypcorr-3}
\end{eqnarray}
where the factor $c$ depends on $\alpha$, $\beta$ and $\gamma$.

As a consequence,
\begin{equation}\label{neq:prasypcoss-4}
	\sum_{n=1}^\infty \corr(X_1,\, (X_{n+1}-X_n)) < \infty
\quad \mbox{and} \quad
	\int_0^\infty \corr(X_1,\, (X_{t+1}-X_t)) \, dt
	< \infty
\end{equation}
if $\min(\gamma,-\frac12)<\alpha-1$, and otherwise
\begin{equation}\label{neq:prasypcoss-5}
	\sum_{n=1}^\infty \corr(X_1,\, (X_{n+1}-X_n)) = \infty
\quad \mbox{and} \quad
	\int_1^\infty \corr(X_1,\, (X_{t+1}-X_t)) \, dt
	= \infty
\end{equation}
if $\alpha\le \frac12$ and $\gamma\ge \alpha-1$.
\end{thm}
\begin{proof}
Relations \eqref{eq:prasypcorr-1} to \eqref{eq:prasypcorr-3}
hold true due to Theorem~\ref{thm:mem1} and Lemma~\ref{lem:incpcovas}.
Thus, either
\begin{gather*}
\corr(X_1,\, (X_{t+1}{-}X_t)) \sim c \, t^{\min(\gamma,-0.5) - \alpha}
\shortintertext{or}
\corr(X_1,\, (X_{t+1}{-}X_t)) \sim c \, t^{\min(\gamma,-0.5) - \alpha}
(\log t)^{-0.5}
\end{gather*}
as $t\to{+}\infty$.
If $\min(\gamma,-0.5)-\alpha < -1$,
the series and the integral in \eqref{neq:prasypcoss-4} are convergent.
If $\min(\gamma,-0.5)-\alpha \ge -1$
(which is equivalent to $\alpha\le\frac12$ \textrm{and}
$\lambda\ge\alpha-1$),
the series and the integral in \eqref{neq:prasypcoss-5} are divergent.
The case where
$\corr(X_t,\, (X_{t+1} - X_t)) \sim c\, t^{-1} (\log t)^{-1/2}$
may need additional explanation. Since
\[
\int t^{-1} (\log t)^{-1/2} \, dt = 2 \sqrt{\log t} + C, \qquad
\int_2^\infty t^{-1} (\log t)^{-1/2} \, dt = \infty,
\]
the integral and the series in \eqref{neq:prasypcoss-5}
are divergent as claimed.
\end{proof}

\begin{remark}
Recall that a  fractional Brownian motion $B^H$ with Hurst index $H\in (0,1)$
is a self-similar zero-mean Gaussian process with stationary increments.
 It can be  normalized in such way that
the unit-length increments have variance equal to 1:
\[
\ME (B^H_1)^2 = \ME (B^H_{t+1}-B^H_t)^2 = 1,
\]
which implies
\[
\cov(B^H_1,\, B^H_{t+1}-B^H_t) =
\corr(B^H_1,\, B^H_{t+1}-B^H_t)
\quad \mbox{for all $t$}.
\]
If $H=0$ or $H=\frac12$, the increments of $B^H$ over non-touching
non-overlapping intervals are independent\footnote{For $H=0$, the process $B^H$ is defined in \cite{bnmz2017}.}.
The process $B^H$ is a process with negatively correlated increments and with short memory if $H\in\bigl (0, \frac12\bigr)$, and
it is a process with positively correlated increments and with long memory if $H\in\bigl (\frac12, 1\bigr]$.
 As to the exponents in the representation \eqref{mainint},
$\beta+\gamma=2H-2>-1$,  $\alpha=\frac12-H \in \bigl(-\frac12, 0\bigr)$
(whence $\alpha\le\frac12$ formally holds true), and
$\gamma=H-\frac32>-1 >\alpha-1$.
These statements are consistent with Theorems
\ref{thm:mem1} and \ref{thm:asypcorr}.
\end{remark}

\begin{remark}
One may ask the following question: do the increments $X_{t_0+t}-X_{t_0}$ approach fBm $B^H(t)$ as $t_0$ tends to infinity?
It turns out that this depends on the values of the parameters $\alpha$, $\beta$ and $\gamma$.
The incremental variance $\var\left (X_{t_0+t}-X_{t_0}\right )$ has non-zero finite limit,
if and only if one of the following conditions is satisfied:
\begin{enumerate}
\item[(a)] \quad $\alpha+\beta =0$ and $\gamma < -1/2$, or
\item[(b)] \quad $\alpha+\beta+\gamma=-1/2$ and $\gamma>-1/2$.
\end{enumerate}

In case (a) the process $\{X_{t_0+t}-X_{t_0},\allowbreak \;
t\in[0,T]\}$
tends to a fractional Brownian motion $c B^H$
with Hurst index $H=\gamma+\frac32$ and $c = (\mathrm{B}(\gamma+1, \: -2\gamma-1) / ((\gamma+1)(2\gamma+1)))^{1/2}$:
\begin{gather*}
  \var\left (X_{t_0+t}-X_{t_0}\right ) \to c^2 \, |t|^{2H} \quad \mbox{as} \quad
  t_0\to+\infty, \\
  \cov\left (X_{t_0+t}-X_{t_0},\: X_{t_0+h+t}-X_{t_0}\right ) \to
  \frac{c^2 (|t|^{2H} + |t+h|^{2H} - |h|^{2H})}{2}
\end{gather*}
as $t_0\to+\infty$.

In case (b) the process 
\[
X_t = \int_0^t s^\alpha \int_s^t u^\beta (u-s)^\gamma du\,dW_s
 = \int_0^t u^\beta \int_0^u  s^\alpha (u-s)^\gamma dW_s\,du
\]
is differentiable a.s., and 
\[
\dot X_t = t^\beta \int_0^t  z^\alpha (t-z)^\gamma dW_z.
\]
Let $s>t$. Then
\[
\ME\dot X_s \dot X_t = s^\beta t^\beta \int_0^t  z^{2\alpha} (t-z)^\gamma (s-z)^\gamma dz,
\]
and consequently, for $y>x$
\begin{align*}
\ME \left(X_y - X_x\right)^2
&= \int_x^y\!\!\int_x^y \ME\dot X_u \dot X_v \,du\,dv
= 2\int_x^y\!\!\int_x^v \ME\dot X_u \dot X_v \,du\,dv
\\
&= 2\int_x^y\!\!\int_x^v u^\beta v^\beta \int_0^u  z^{2\alpha} (u-z)^\gamma (v-z)^\gamma dz \,du\,dv.
\end{align*}
Now, assume that $x\to\infty$ and $y=x+h$ for some fixed $h>0$ (the case $h<0$ can be considered similarly).
Then
\begin{equation}\label{eq:no}
\begin{split}
\ME \left(X_{x+h} - X_x\right)^2
&= 2\int_x^{x+h}\!\!\int_x^v u^\beta v^\beta \int_0^u  z^{2\alpha} (u-z)^\gamma (v-z)^\gamma dz \,du\,dv
\\
&= 2\int_x^{x+h}\!\!\int_x^v u^\beta v^\beta u^{2\alpha+2\gamma+1}\int_0^1  t^{2\alpha} \left (\tfrac{v}{u}-t\right )^\gamma (1-t)^\gamma dt \,du\,dv
\\
&= 2x^2\int_1^{1+h/x}\!\!\int_1^r p^{2\alpha+2\gamma+\beta+1} r^\beta \int_0^1  t^{2\alpha} \left (\tfrac{r}{p}-t\right )^\gamma (1-t)^\gamma dt \,dp\,dr.
\end{split}
\end{equation}

Let $x\to\infty$. Then $\frac{r}{p}\downarrow1$,
\[
\int_0^1  t^{2\alpha} \left (\tfrac{r}{p}-t\right )^\gamma (1-t)^\gamma dt \to B(2\alpha+1,2\gamma+1).
\]
Therefore, the last integral in \eqref{eq:no} has the following asymptotic behavior
\begin{align*}
\MoveEqLeft
B(2\alpha+1,2\gamma+1)\cdot 2x^2\int_1^{1+h/x}\!\!\int_1^r p^{-\beta} r^\beta \,dp\,dr
\\*
&=B(2\alpha+1,2\gamma+1) \frac{2x^2}{1-\beta}\int_1^{1+h/x}\left(r^{1-\beta}-1\right) r^\beta \,dr
\\
&=B(2\alpha+1,2\gamma+1) \frac{2x^2}{1-\beta}\int_1^{1+h/x}\left(r- r^\beta\right)dr
\\
&=B(2\alpha+1,2\gamma+1) \frac{2x^2}{1-\beta}\left(\frac12\left (1+\frac{h}{x}\right )^2 -\frac12-\frac{(1+\frac{h}{x})^{\beta+1}-1}{1+\beta}\right)
\\
&=B(2\alpha+1,2\gamma+1) \frac{2x^2}{1-\beta}\left(\frac12\frac{h^2}{x^2}-\frac\beta2\frac{h^2}{x^2}+o\left (\frac{1}{x^2}\right )\right)
\\
&\to B(2\alpha+1,2\gamma+1)h^2, \quad x\to\infty.
\end{align*}
Thus, the limit process can be tractable (up to a constant) as a fractional Brownian motion with $H=1$, i.\,e., as $B^1_t = t\xi$ with $\xi\sim\mathcal N(0,1)$.
\end{remark}

\section{Drift parameter estimation in the Ornstein--\newline Uhlenbeck process driven by Gaussian Volterra process}\label{s:4}

Let inequalities \eqref{neq:procXconds} hold true, and let the process $X$ be defined in \eqref{mainint}.
In this section we consider the estimation of the unknown parameter $\theta>0$ by observations of the process $Z=\{Z_t,t\geq 0\}$ that is a solution of the following stochastic differential equation of Langevin type
\begin{equation}\label{eq:equation}
Z_t=\theta\int_0^tZ_s\,ds+X_t.
\end{equation}

In what follows we make use of the following two properties of the process $X$. Firstly, recall that according to \cite[Thm.\ 6]{Part1},
\begin{enumerate}[(H1)]
\item \label{(H1)}
the process $X$ has a continuous modification that is H\"older continuous up to order
$\lambda=\min\left (\alpha+\beta+\gamma+\frac32, \gamma+\frac32, 1\right )$.
\end{enumerate}
In what follows we shall consider this modification.
Secondly, according to \cite[Prop.\ 1]{Part1}, the process $X$ is self-similar with exponent $\alpha+\beta+\gamma+\frac32$, hence,
\begin{enumerate}[resume*]
\item \label{(H2)} for every $t \ge 0$, $\ME X^2_t = C t^{2\alpha+2\beta+2\gamma+3}$, where $C=\ME X_1^2<\infty$.
\end{enumerate}

One can verify that the solution of~\eqref{eq:equation} can be  represented as
\begin{equation}\label{eq:solution}
Z_t=e^{\theta t}\int_0^te^{-\theta s}\,dX_s, \quad t\ge0.
\end{equation}
Due to the H\"older continuity of $X$ and the Lipschitz continuity
of the function $s\mapsto e^{-\theta s}$, the integral with respect to $X$ in \eqref{eq:solution} exists
as the pathwise Riemann--Stieltjes integral.

Let, more precisely, our goal be to estimate the unknown drift parameter $\theta>0$ by the continuous-time observations on the interval $[0,T]$.
Consider the estimator
\begin{equation}\label{eq:estimator1}
\hat\theta_T=\frac{Z_T^2}{2\int_0^TZ_s^2\,ds}.
\end{equation}

\begin{remark}\label{rem:holder}
As mentioned above, the trajectories of the processes $X$ and consequently $Z$ are a.\,s.\ H\"older continuous up to order $\lambda$.
Therefore, if $\lambda>\frac12$, then the path-wise integral $\int_0^TZ_s\,dZ_s$ is well defined as Young's integral and equals $Z_T^2/2$, hence, $\hat\theta_T$ can be written as
$\hat\theta_T=\frac{\int_0^TZ_s\,dZ_s}{\int_0^TZ_s^2\,ds}$.
Hence, the $\hat\theta_T$ coincides with the least squares estimator, which  minimizes (formally) the integral $\theta\mapsto\int_0^T \left|\dot Z_s - \theta Z_s\right|^2ds$
(we refer to \cite{emeso2016} for details).
However in what follows we will consider the estimator $\hat\theta_T$ in the form \eqref{eq:estimator1}, since this form is well defined for all values of $\lambda\in(0,1]$.
\end{remark}

\begin{thm}
Let $\theta>0$.
Then the estimator $\hat\theta_T$ is strongly consistent, that is
\[
\hat\theta_T\to\theta \quad\text{a.s., as } T\to\infty.
\]
\end{thm}

\begin{proof}
Note that $Z$ is an Ornstein--Uhlenbeck process driven by a centered Gaussian process $X$ with $X_0=0$.
The drift parameter estimation for such models was investigated in \cite{emeso2016}. In particular, the strong consistency of $\hat\theta_T$ follows immediately from \cite[Thm.~2.1]{emeso2016}, whose assumptions are satisfied due to the properties \ref{(H1)} and~\ref{(H2)}.
Note that the proof of \cite[Thm.~2.1]{emeso2016} implicitly assumes that $e^{-\theta T} X_T \to 0$ a.s.\ as $T\to\infty$ (this condition is needed to derive the convergence (2.11) in \cite{emeso2016} from (2.8) and (2.10)).
In our case this assumption also holds due to Theorem~\ref{thm:2.3}.
\end{proof}

Our next goal is to prove that $\hat \theta_T$ has a Cauchy-type asymptotic distribution.
Denote
\[
\eta_t \coloneqq \int_0^t e^{-\theta s} X_s \,ds,
\quad t\ge0.
\]
Then the random variable
$\eta_\infty \coloneqq \int_0^\infty e^{-\theta s} X_s \,ds$
is well defined, and moreover,
$\eta_t \to \eta_\infty$ as $t\to\infty$
a.\,s. and in $L_2(\Omega)$.

Let $\dto$ denote the weak convergence in distribution and let $\cauchy$ be the standard Cauchy distribution with the probability density function
$\frac{1}{\pi(1+x^2)}$, $x \in \mathbb R$.
The next theorem is the main result of this section.

\begin{thm}\label{th:cauchy}
Let $\alpha$, $\beta$ and $\gamma$ satisfy the conditions \eqref{neq:procXconds}, and let $\theta>0$.

\begin{enumerate}
\item
If $\gamma\in(-1,-\frac12)$, then
\[
\frac{e^{\theta T}}{T^{\alpha+\beta}} \left(\hat\theta_T-\theta\right )
\dto \frac{2\sigma}{\theta^{\gamma+3/2}\sqrt{\ME \eta_\infty^2}} \, \cauchy,
\quad\text{as } T\to\infty,
\]
where
$\sigma = \sigma(\gamma) = \sqrt{\frac{\Gamma(\gamma+1)\Gamma(-2\gamma-1)\Gamma(2\gamma+2)}{\Gamma(-\gamma)}}$.
\item
If $\gamma = -\frac12$, then
\[
\frac{e^{\theta T}}{T^{\alpha+\beta} \sqrt{\log T}} \left(\hat\theta_T-\theta\right )
\dto \frac{2}{\theta \sqrt{\ME \eta_\infty^2}} \, \cauchy,
\quad\text{as } T\to\infty.
\]

\item
If $\gamma > -\frac12$, then
\[
\frac{e^{\theta T}}{T^{\alpha+\beta+\gamma+1/2}} \left(\hat\theta_T-\theta\right )
\dto \frac{2\varsigma}{\theta \sqrt{\ME \eta_\infty^2}} \, \cauchy,
\quad\text{as } T\to\infty,
\]
where
$\varsigma = \varsigma(\alpha,\gamma) = \sqrt{\frac{\Gamma(2\alpha+1)\Gamma(2\gamma+1)}{\Gamma(2\alpha+2\gamma + 2)}}$.

\end{enumerate}
\end{thm}

\begin{remark}
The asymptotic distribution of $\hat\theta_T$ when $X$ is a fractional Brownian motion was first obtained in \cite{belfadli2011}.
In this case $\alpha = \frac12 - H$, $\beta = H - \frac12$,  $\gamma = H - \frac32$ for $H\in(\frac12,1)$, so it corresponds to the first statement of Theorem~\ref{th:cauchy}.
\end{remark}

The proof of Theorem~\ref{th:cauchy} follows the scheme from \cite{emeso2016}. We split it into several lemmata. First, let us note that similarly to the proof of \cite[Thm.~2.2]{emeso2016}, we have the following representation for $e^{\theta T} \left(\hat\theta_T-\theta\right )$.

\begin{lemma}[\cite{emeso2016}]
For any $T>0$,
\begin{equation} \label{eq:err-repr}
e^{\theta T} \left(\hat\theta_T-\theta\right )
= a_T b_T + c_T,
\end{equation}
where
\[
a_T = \frac{e^{-\theta T} \int_0^T e^{\theta s}\,dX_s}{\eta_\infty},
\qquad
b_T = \frac{\theta \eta_T \eta_\infty}{e^{-2\theta T} \int_0^T Z_s^2\,ds},
\qquad
c_T = \frac{e^{-\theta T} R_T}{e^{-2\theta T} \int_0^T Z_s^2\,ds},
\]
$R_T = \frac12 X_T^2 - \theta \int_0^T X_s^2\,ds
+ \theta^2 \int_0^T\!\!\int_0^s X_s X_r e^{-\theta(s-r)}\,dr\,ds$.
Moreover, the following convergences hold as $T\to \infty$
\begin{equation}\label{eq:b-c-conv}
b_T \to 2\;\text{ a.s.},
\qquad
c_T \to 0 \;\text{ in probability}.
\end{equation}
\end{lemma}

Indeed, the analysis of the proof of \cite[Thm.~2.2]{emeso2016} shows that the conditions (H1) and (H2) are sufficient for the representation \eqref{eq:err-repr}, as well as for convergences \eqref{eq:b-c-conv}.
Hence, it remains to investigate the asymptotic behavior of $a_T$. And the remaining part of this section is devoted to this issue.
Or first goal is to study the asymptotic behavior of
the expectations
\begin{equation}\label{eq:expect}
\ME\left(e^{-\theta T} \int_0^T e^{\theta s}dX_s\right )^2
\quad\text{and}\quad
\ME\left(X_s e^{-\theta T} \int_0^T e^{\theta r}dX_r\right ),
\; s \ge 0.
\end{equation}

\begin{remark}
Note that, generally speaking, the expectations \eqref{eq:expect} do not satisfy the conditions $(\mathcal{H}3)$ and $(\mathcal{H}4)$ of the paper \cite{emeso2016}. In particular, the first expectation does not converge to a positive constant (see Lemma~\ref{l:H3-l4} below). Therefore the results of \cite{emeso2016} (in particular, Lemma 2.4) cannot be applied to establish the asymptotic behavior of $a_T$ directly.
\end{remark}

The following lemma describes the asymptotic behavior of the first expectation in \eqref{eq:expect}.
This result can be viewed as an analog of the assumption $(\mathcal H3)$ from \cite{emeso2016}.
\begin{lemma}\label{l:H3-l4}
Let $\alpha$, $\beta$ and $\gamma$ satisfy the conditions \eqref{neq:procXconds}, and let $\theta>0$.

\begin{enumerate}
\item
If $\gamma\in(-1,-\frac12)$, then
\[
\ME\left(e^{-\theta T} \int_0^T e^{\theta s}dX_s\right )^2
\sim \frac{\sigma^2}{\theta^{2\gamma+3}}\, T^{2\alpha+2\beta},
\quad\text{as } T\to\infty.
\]
\item
If $\gamma = -\frac12$, then
\[
\ME\left(e^{-\theta T} \int_0^T e^{\theta s}dX_s\right )^2
\sim \frac{1}{\theta^2} T^{2\alpha+2\beta} \log T,
\quad\text{as } T\to\infty.
\]
\item
If $\gamma > -\frac12$, then
\[
\ME\left(e^{-\theta T} \int_0^T e^{\theta s}dX_s\right )^2
\sim \frac{\varsigma^2}{\theta^2} T^{2\alpha+2\beta+2\gamma+1},
\quad\text{as } T\to\infty.
\]
\end{enumerate}
\end{lemma}

\begin{proof}
Using formula (14) from \cite{Part2} for the integral w.r.t.\ $X$ we may write
\begin{align*}
\MoveEqLeft
\ME\left(e^{-\theta T} \int_0^T e^{\theta s}dX_s\right )^2
=  e^{-2\theta T} \ME\left( \int_0^T s^\alpha \int_s^T e^{\theta t}t^{\beta} (t-s)^\gamma\,dt\,dW_s\right )^2
\\
&= e^{-2\theta T} \int_0^T s^{2\alpha}\left(\int_s^T e^{\theta t}t^{\beta} (t-s)^\gamma\,dt\right )^2 ds.
\end{align*}

1. Let $\gamma\in(-1,-\frac12)$.
Then using l'H\^opital's rule  we obtain
\begin{align*}
\MoveEqLeft
\lim_{T\to \infty} T^{-2\alpha-2\beta}\ME\left(e^{-\theta T} \int_0^T e^{\theta s}dX_s\right )^2
=\lim_{T\to \infty}
\frac{\int_0^T s^{2\alpha}\left(\int_s^T e^{\theta t}t^{\beta} (t-s)^\gamma\,dt\right )^2 ds}{T^{2\alpha+2\beta}e^{2\theta T}}
\\
&=\lim_{T\to \infty}
\frac{\int_0^T s^{2\alpha}\cdot 2\left(\int_s^T e^{\theta t}t^{\beta} (t-s)^\gamma\,dt\right ) e^{\theta T} T^{\beta} (T-s)^\gamma ds}{2\theta T^{2\alpha+2\beta}e^{2\theta T} + (2\alpha+2\beta) T^{2\alpha+2\beta-1}e^{2\theta T}}
\\
&=\lim_{T\to \infty}
\frac{\int_0^T s^{2\alpha}\left(\int_s^T e^{\theta t}t^{\beta} (t-s)^\gamma\,dt\right ) (T-s)^\gamma ds}{\theta T^{2\alpha+\beta}e^{\theta T} + (\alpha+\beta) T^{2\alpha+\beta-1}e^{\theta T}}
=\frac{\sigma^2}{\theta^{2\gamma+3}},
\end{align*}
where the last equality follows from Lemma \ref{l:H3-l3} (Appendix~\ref{s:B}).

2. Let $\gamma=-\frac12$.
Then by l'H\^opital's rule and Lemma \ref{l:H3-l3} (Appendix~\ref{s:B}), we get
\begin{align*}
\MoveEqLeft
\lim_{T\to \infty} \frac{1}{T^{2\alpha+2\beta}\log T}\ME\left(e^{-\theta T} \int_0^T e^{\theta s}dX_s\right )^2
\\
&=\lim_{T\to \infty}
\frac{\int_0^T s^{2\alpha}\left(\int_s^T e^{\theta t}t^{\beta} (t-s)^{-\frac12}\,dt\right )^2 ds}{T^{2\alpha+2\beta}(\log T)e^{2\theta T}}
\\
&=\lim_{T\to \infty}
\frac{\int_0^T s^{2\alpha}\cdot 2\left(\int_s^T e^{\theta t}t^{\beta} (t-s)^{-\frac12}\,dt\right ) e^{\theta T} T^{\beta} (T-s)^{-\frac12} ds}{2\theta T^{2\alpha+2\beta}e^{2\theta T}\log T + (2\alpha+2\beta) T^{2\alpha+2\beta-1}e^{2\theta T}\log T + T^{2\alpha+2\beta-1}e^{2\theta T}}\\
&=\frac{1}{\theta^2}.
\end{align*}

3. Let $\gamma>-\frac12$.
Then arguing as above we have
\begin{align*}
\MoveEqLeft
\lim_{T\to \infty} \frac{1}{T^{2\alpha+2\beta+2\gamma+1}}\ME\left(e^{-\theta T} \int_0^T e^{\theta s}dX_s\right )^2
\\
&=\lim_{T\to \infty}
\frac{\int_0^T s^{2\alpha}\left(\int_s^T e^{\theta t}t^{\beta} (t-s)^\gamma\,dt\right )^2 ds}{T^{2\alpha+2\beta+2\gamma+1} e^{2\theta T}}
\\
&=\lim_{T\to \infty}
\frac{\int_0^T s^{2\alpha}\cdot 2\left(\int_s^T e^{\theta t}t^{\beta} (t-s)^\gamma\,dt\right ) e^{\theta T} T^{\beta} (T-s)^\gamma ds}{2\theta T^{2\alpha+2\beta+2\gamma+1}e^{2\theta T} + (2\alpha+2\beta+2\gamma+1) T^{2\alpha+2\beta+2\gamma}e^{2\theta T}}
=\frac{\varsigma^2}{\theta^2}.
\qedhere
\end{align*}
\end{proof}

The asymptotic behavior of the second expectation in \eqref{eq:expect} is stated in the following result.
\begin{lemma}\label{l:H4}
Let $\alpha$, $\beta$ and $\gamma$ satisfy the conditions \eqref{neq:procXconds}, and let $\theta>0$.
Then for any $s\ge0$
\[
\ME\left(X_s e^{-\theta T} \int_0^T e^{\theta r}dX_r\right )
\sim \frac{B(2\alpha+1,\gamma+1) s^{2\alpha+\beta+\gamma+2}}{\theta (2\alpha+\beta+\gamma+2)} \, T^{\beta+\gamma},
\quad\text{as } T\to\infty.
\]
\end{lemma}

\begin{proof}
By formula (14) from \cite{Part2} for the integral w.r.t.\ $X$, we have for $0<s< T$
\begin{align*}
\MoveEqLeft
\ME\left(X_s e^{-\theta T} \int_0^T e^{\theta r}dX_r\right )
\\
&= e^{-\theta T} \ME\left(\int_0^s r^{\alpha} \int_r^s u^{\beta} (u-r)^{\gamma}\,du\,dW_r \cdot  \int_0^T r^\alpha \int_r^T e^{\theta t}t^{\beta} (t-r)^\gamma\,dt\,dW_r\right )
\\
&= e^{-\theta T} \int_0^s r^{2\alpha} \int_r^s u^{\beta} (u-r)^{\gamma}\,du \int_r^T e^{\theta t}t^{\beta} (t-r)^\gamma\,dt\,dr
\end{align*}
Then using this representation and l'H\^opital's rule, we get
\begin{align*}
\MoveEqLeft
\lim_{T\to\infty} T^{-\beta-\gamma}\ME\left(X_s e^{-\theta T} \int_0^T e^{\theta r}dX_r\right )
\\
&= \lim_{T\to\infty}  \frac{\int_0^s r^{2\alpha} \int_r^s u^{\beta} (u-r)^{\gamma}\,du \int_r^T e^{\theta t}t^{\beta} (t-r)^\gamma\,dt\,dr}{e^{\theta T} T^{\beta+\gamma}}
\\
&=\lim_{T\to\infty} \frac{\int_0^s r^{2\alpha} \int_r^s u^{\beta} (u-r)^{\gamma}\,du \, e^{\theta T}T^{\beta} (T-r)^\gamma\,dr}{\theta e^{\theta T}T^{\beta+\gamma} + e^{\theta T}(\beta+\gamma)T^{\beta+\gamma-1}}
\\
&=\lim_{T\to\infty} \frac{\int_0^s r^{2\alpha} (1-\frac{r}{T})^\gamma\int_r^s u^{\beta} (u-r)^{\gamma}\,du \,dr}{\theta + (\beta+\gamma)T^{-1}}.
\end{align*}
Finally, applying the Lebesgue dominated convergence theorem in the numerator, we obtain
\begin{align*}
\MoveEqLeft
\lim_{T\to\infty} T^{-\beta-\gamma}\ME\left(X_s e^{-\theta T} \int_0^T e^{\theta r}dX_r\right )
= \frac{1}{\theta} \int_0^s r^{2\alpha}\int_r^s u^{\beta} (u-r)^{\gamma}\,du\,dr
\\
&= \frac{1}{\theta} \int_0^s u^{\beta} \int_0^u r^{2\alpha}  (u-r)^{\gamma}\,dr\,du
= \frac{1}{\theta} B(2\alpha+1,\gamma+1)\int_0^s u^{2\alpha+\beta+\gamma+1} \,du
\\
&= \frac{B(2\alpha+1,\gamma+1) s^{2\alpha+\beta+\gamma+2} }{\theta (2\alpha+\beta+\gamma+2)}.
\qedhere
\end{align*}
\end{proof}

%
%

\begin{lemma}\label{l:a-conv}
Assume that the conditions \eqref{neq:procXconds} hold.
Let $F$ be any $\sigma\{X\}$-measurable random variable such that
$\prob(F < \infty) = 1$.
Denote by $N$ the standard normal random variable independent of $X$.
\begin{enumerate}
\item
If $\gamma\in(-1,-\frac12)$, then
\[
\left(F, \frac{1}{T^{\alpha+\beta}}\,e^{-\theta T} \int_0^T e^{\theta s}dX_s\right)
\dto \left(F, \frac{\sigma}{\theta^{\gamma+3/2}} N\right),
\quad\text{as } T\to\infty.
\]

\item
If $\gamma = -\frac12$, then
\[
\left(F, \frac{1}{T^{\alpha+\beta} \sqrt{\log T}}\,e^{-\theta T} \int_0^T e^{\theta s}dX_s\right)
\dto \left(F, \frac{1}{\theta} N\right),
\quad\text{as } T\to\infty.
\]

\item
If $\gamma > -\frac12$, then
\[
\left(F, \frac{1}{T^{\alpha+\beta+\gamma+1/2}}\,e^{-\theta T} \int_0^T e^{\theta s}dX_s\right)
\dto \left(F, \frac{\varsigma}{\theta} N\right),
\quad\text{as } T\to\infty.
\]
\end{enumerate}
\end{lemma}

\begin{proof}
All three statements of the lemma are derived from Lemmata~\ref{l:H3-l4} and \ref{l:H4}. For example, let us consider the first statement.
As explained in the proof of \cite[Lemma~2.4]{emeso2016}, it suffices to prove that for any $d\ge1$ and any $s_1,\dots,s_d\in[0,\infty)$, as  $T\to\infty$,
\[
\left(X_{s_1},\dots, X_{s_d}, \frac{1}{T^{\alpha+\beta}}\,e^{-\theta T} \int_0^T e^{\theta s}dX_s\right)
\dto \left(X_{s_1},\dots, X_{s_d}, \frac{\sigma}{\theta^{\gamma+3/2}} N\right),
\]
and moreover, due to Gaussianity, it is sufficient to verify the convergence of the corresponding covariance matrices.
In turn, this convergence follows from Lemmata~\ref{l:H3-l4} and \ref{l:H4}. The second and the third statements of the lemma are proved by similar arguments.
\end{proof}

From Lemma~\ref{l:a-conv} one can derive the asymptotic behavior of the term $a_T$ defined in \eqref{eq:err-repr}.
\begin{corollary}\label{cor:a-conv}
Let the conditions of Theorem~\ref{th:cauchy} hold.
\begin{enumerate}
\item
If $\gamma\in(-1,-\frac12)$, then
\[
\frac{a_T}{T^{\alpha+\beta}}
\dto \frac{\sigma}{\theta^{\gamma+3/2}} \frac{N}{\eta_\infty}
\deq \frac{\sigma}{\theta^{\gamma+3/2}\sqrt{\ME\eta_\infty^2}} \, \cauchy,
\quad\text{as } T\to\infty.
\]

\item
If $\gamma = -\frac12$, then
\[
\frac{a_T}{T^{\alpha+\beta} \sqrt{\log T}}
\dto \frac{1}{\theta} \frac{N}{\eta_\infty}
\deq \frac{1}{\theta\sqrt{\ME\eta_\infty^2}} \, \cauchy,
\quad\text{as } T\to\infty.
\]

\item
If $\gamma > -\frac12$, then
\[
\frac{a_T}{T^{\alpha+\beta+\gamma+1/2}}
\dto \frac{\varsigma}{\theta^{\gamma+1/2}} \frac{N}{\eta_\infty}
\deq \frac{\varsigma}{\theta^{\gamma+1/2}\sqrt{\ME\eta_\infty^2}} \, \cauchy,
\quad\text{as } T\to\infty.
\]
\end{enumerate}

\end{corollary}

\begin{proof}[Proof of Theorem~\ref{th:cauchy}]
Now the proof of Theorem~\ref{th:cauchy} immediately follows from the representation \eqref{eq:err-repr}, the convergences \eqref{eq:b-c-conv} and Corollary~\ref{cor:a-conv}, by the Slutsky theorem.
\end{proof}

\begin{remark}
The limits in Theorem~\ref{th:cauchy} contain the second moment $\ME\eta_\infty^2$. It can be computed by the following formulae:
\begin{equation}\label{eq:sec-mom}
\begin{split}
\ME\eta_\infty^2
&=\frac{2\Gamma(2\alpha+2\beta+2\gamma+3)}{\theta^{2\alpha+2\beta+2\gamma+5}}\cdot\frac{\Gamma(2\alpha+1)\Gamma(\gamma+1)}{\Gamma(2\alpha+\gamma+2)}
\\*
&\quad\times\int_0^1 \frac{x^{2\alpha+\beta+\gamma+1}}{(1+x)^{2\alpha+2\beta+2\gamma+3}}
\,\hyper(-\gamma,2\alpha+1;2\alpha+\gamma+2;x)\,dx,
\end{split}
\end{equation}
where $\hyper$ denotes the Gauss hypergeometric function (see Appendix~\ref{s:B}).
Indeed, using integration by parts, we can represent $\eta_\infty$ as follows
\[
\eta_\infty = -\frac1\theta\int_0^\infty  X_s \,d e^{-\theta s}
 = -\frac1\theta X_s e^{-\theta s}\Big|_{s=0}^\infty+\frac1\theta\int_0^\infty e^{-\theta s} dX_s
 =\frac1\theta\int_0^\infty e^{-\theta s} dX_s.
\]
Here the a.s.\ convergence $X_s e^{-\theta s}\to0$ as $s\to\infty$ follows from the bound \eqref{neq:thm-ssGrow1-asj-a}.
By \cite[formula (14)]{Part2}, the second moment of the integral w.r.t.\ X can be written as
\begin{align*}
\ME\eta_\infty^2
&=\frac{1}{\theta^2}\int_0^\infty s^{2\alpha}\left(\int_s^\infty e^{-\theta t} t^\beta (t-s)^\gamma\,dt\right)^2 ds
\\
&=\frac{1}{\theta^2}\int_0^\infty\!\!\int_s^\infty\!\!\int_s^\infty s^{2\alpha} e^{-\theta t-\theta r} t^\beta (t-s)^\gamma r^\beta (r-s)^\gamma\,dr\,dt\,ds
\\
&=\frac{1}{\theta^2}\int_0^\infty\!\!\int_0^t\!\!\int_s^\infty s^{2\alpha} e^{-\theta t-\theta r} t^\beta  (t-s)^\gamma r^\beta (r-s)^\gamma\,dr\,ds\,dt
\\
&=\frac{1}{\theta^2}\int_0^\infty\!\!\int_0^t\!\!\int_s^t s^{2\alpha} e^{-\theta t-\theta r} t^\beta  (t-s)^\gamma r^\beta (r-s)^\gamma\,dr\,ds\,dt
\\
&\quad+\frac{1}{\theta^2}\int_0^\infty\!\!\int_0^t\!\!\int_t^\infty s^{2\alpha} e^{-\theta t-\theta r} t^\beta  (t-s)^\gamma r^\beta (r-s)^\gamma\,dr\,ds\,dt
\\
&=\frac{1}{\theta^2}\int_0^\infty\!\!\int_0^t\!\!\int_0^r s^{2\alpha} e^{-\theta t-\theta r} t^\beta  (t-s)^\gamma r^\beta (r-s)^\gamma\,ds\,dr\,dt
\\
&\quad+\frac{1}{\theta^2}\int_0^\infty\!\!\int_t^\infty\!\!\int_0^t s^{2\alpha} e^{-\theta t-\theta r} t^\beta  (t-s)^\gamma r^\beta (r-s)^\gamma\,ds\,dr\,dt
\\
&=\frac{2}{\theta^2}\int_0^\infty\!\!\int_0^t\!\!\int_0^r s^{2\alpha} e^{-\theta t-\theta r} t^\beta  (t-s)^\gamma r^\beta (r-s)^\gamma\,ds\,dr\,dt
\\
&=\frac{2}{\theta^2}\int_0^\infty\!\!\int_0^t\!\!\int_0^1 z^{2\alpha} e^{-\theta t-\theta r} t^\beta  (t-zr)^\gamma r^{2\alpha+\beta+\gamma+1} (1-z)^\gamma\,dz\,dr\,dt
\\
&=\frac{2}{\theta^2}\int_0^\infty\!\!\!\int_0^1\!\!\int_0^1 z^{2\alpha} e^{-\theta t(1+x)} t^{2\alpha+2\beta+2\gamma+1} (1-zx)^\gamma x^{2\alpha+\beta+\gamma+1} (1-z)^\gamma dz dx dt
\\
&=\frac{2\Gamma(2\alpha+2\beta+2\gamma+3)}{\theta^{2\alpha+2\beta+2\gamma+5}}\int_0^1\!\!\int_0^1 \frac{z^{2\alpha}(1-zx)^\gamma x^{2\alpha+\beta+\gamma+1} (1-z)^\gamma}{(1+x)^{2\alpha+2\beta+2\gamma+3}}\,dz\,dx,
\end{align*}
whence \eqref{eq:sec-mom} follows due to \eqref{eq:hypergeom-int}.
\end{remark}

\begin{appendix}
\section{Auxiliary results}\label{s:A}
\label{apx:apx}

In the following definition, $\mathrm{T}$ is an index set of a stochastic
process; either $\mathrm{T}=[0,+\infty)$ or $\mathrm{T}=(0,+\infty)$.

\begin{definition}
The stochastic process $\{X_t,\; t\in\mathrm{T}\}$
is called  self-similar  (as defined in \cite{Embrechts2002})
if for every $a>0$ there exists $b>0$ such that the processes
$\{X_{at},\;\allowbreak t \in\mathrm{T}\}$ and
$\{b X_t,\;\allowbreak t\in\mathrm{T}\}$
have the same distribution.

The stochastic process $\{X_t,\; t\in\mathrm{T}\}$
is called  self-similar with self-similarity exponent  $\rho$
if for every $a>0$ the processes
$\{X_{at},\;\allowbreak t \in\mathrm{T}\}$ and
$\{a^\rho X_t,\;\allowbreak t\in\mathrm{T}\}$
have the same distribution.
\end{definition}

\begin{prop}[Lamperti \cite{Lamperti1962}]
\label{prop:sstostat-part1}
Let $\{X_t, t>0\}$
be a self-similar zero-mean Gaussian process
with self-similarity exponent $\rho$.
Then
$\{e^{-\rho t} X_{e^t}, t\in\mathbb{R}\}$
is a stationary zero-mean Gaussian process.
\end{prop}

\begin{thm}[Marcus {\cite[Theorem~1.4]{Marcus1972}}]\label{thm:Marcus-theorem}
Let $X$ be a continuous stationary Gaussian process with $\ME X_t^2=1$.
Then
\[
\limsup_{t\to +\infty} \frac{|X_t|}{(2 \log t)^{1/2}}
\le 1
\quad  \text{almost surely}.
\]
\end{thm}

The following result is Theorem~1.1.1 from \cite{Embrechts2002}.
\begin{thm}[Embrechts and Maejima \cite{Embrechts2002}]
\label{thm:Hnonneg}
Let the stochastic process $\{X_t,\allowbreak t \ge0\}$
be nontrivial,
self-similar (as defined in \cite{Embrechts2002}),
and stochastically continuous at point $t=0$
Then $\{X_t,\allowbreak t\ge0\}$
is  self-similar with a  unique  exponent  $\rho$, and
the self-similarity exponent $\rho$ is nonnegative.
\end{thm}

\begin{remark}
The continuity condition is essential in Theorem~\ref{thm:Hnonneg},
as it is shown in the following examples.

The process $\{\dot X_t, t\ge 0\}$ defined in \cite{Part2} by formula
\[
\dot X_t = \int_0^t s^\alpha t^\beta (t-s)^\gamma \, dW_s
\]
for $\alpha>-\frac12$, $\beta\in\mathbb{R}$ and $\gamma>-\frac12$,
is self-similar with exponent
$\alpha+\beta+\gamma+\frac12$.
The process $\dot X$ is continuous at 0 if $\alpha+\beta+\gamma+\frac12>0$,
and it is not continuous at 0 if $\alpha+\beta+\gamma+\frac12<0$.

A non-stochastic process $\{X_t,\; t>0\}$ with $X(t)=t^\rho$
is self-similar with exponent $\rho$.
If $\rho\ge 0$, it can be continuously extended to point $0$.
Otherwise, if $\rho < 0$, such extension is impossible.

Notice also that Embrechts and Maejima's \cite{Embrechts2002} definition of self-similarity is
too general for nowhere-continuous processes.
The process that satisfies \cite[Definition 1.1.1]{Embrechts2002}
does not need to have the self-similarity exponent.
\end{remark}

\section{Gauss hypergeometric function and related integrals}\label{s:B}
In this appendix we collect the definition of the Gauss hypergeometric function $\hyper$ and some its properties that are required for our proofs.
We refer to the book \cite{Andrews1999} for further information on this topic.

Although $\hyper(a,b;c;x)$ can be defined for complex $a$, $b$, $c$ and $x$, here
we restrict ourselves to the case of real arguments.
Moreover, we assume that $c>b>0$.
In this case, we may define $\hyper(a,b;c;x)$ for $x<1$ by
the following Euler's integral representation \cite[Thm.~2.2.1]{Andrews1999}:
\begin{equation}\label{eq:hypergeom-int}
\hyper(a,b;c;x) = \frac{\Gamma(c)}{\Gamma(b)\Gamma(c-b)}
\int_0^1 t^{b-1} (1-t)^{c-b-1} (1-xt)^{-a} \,dt.
\end{equation}
The next result describes the behavior of the hypergeometric function as $x\uparrow1$.
\begin{prop}[{\cite[Thm.~2.1.3 and Thm.~2.2.2]{Andrews1999}}]
\label{prop:hyperlim1}
\leavevmode
\begin{enumerate}[(i)]
\item \label{prop-i}
If $c<a+b$, then \
$\displaystyle \lim_{x\uparrow1}\frac{\hyper(a,b;c;x)}{(1-x)^{c-a-b}}
= \frac{\Gamma(c)\Gamma(a+b-c)}{\Gamma(a)\Gamma(b)}$.

\item \label{prop-ii}
If $c=a+b$, then \
$\displaystyle
\lim_{x\uparrow1}\frac{\hyper(a,b;a+b;x)}{\log(1/(1-x))}
= \frac{\Gamma(a+b)}{\Gamma(a)\Gamma(b)}$.

\item \label{prop-iii}
If $c>a+b$, then \
$\displaystyle
\hyper(a,b;c;1) = \frac{\Gamma(c)\Gamma(c-a-b)}{\Gamma(c-a)\Gamma(c-b)}$.
\end{enumerate}
\end{prop}

This result allows us to investigate the asymptotic behavior of the double integral from the proof of Lemma~\ref{l:H3-l4}.
\begin{lemma}\label{l:H3-l3}
Let $\alpha$, $\beta$ and $\gamma$ satisfy the conditions \eqref{neq:procXconds}, and let $\theta>0$.
Denote
\[
h(T) = \int_0^T s^{2\alpha}(T-s)^\gamma\int_s^T e^{\theta t}t^{\beta} (t-s)^\gamma\,dt\, ds, \quad T>0.
\]
\begin{enumerate}
\item
If $\gamma\in(-1,-\frac12)$, then
\begin{equation}\label{eq:h1}
\frac{h(T)}{T^{2\alpha+\beta} e^{\theta T}} \to \frac{\sigma^2}{\theta^{2\gamma+2}},
\quad\text{as } T\to\infty.
\end{equation}
\item
If $\gamma = -\frac12$, then
\[
\frac{h(T)}{T^{2\alpha+\beta} e^{\theta T}\log T} \to  \frac{1}{\theta},
\quad\text{as } T\to\infty.
\]
\item
If $\gamma > -\frac12$, then
\[
\frac{h(T)}{T^{2\alpha+\beta+2\gamma+1} e^{\theta T}} \to \frac{\varsigma^2}{\theta},
\quad\text{as } T\to\infty.
\]
\end{enumerate}
\end{lemma}
\begin{proof}
Changing the order of integration and then making the substitutions $t=T-x$ and $s=(T-x) y$, we transform $h(T)$ as follows:
\begin{align}
h(T)
&=\int_0^T e^{\theta t}t^{\beta} \int_0^t s^{2\alpha}(T-s)^\gamma  (t-s)^\gamma\, ds\,dt
\notag\\
&=\int_0^T e^{\theta (T-x)}(T-x)^{\beta} \int_0^{T-x} s^{2\alpha}(T-s)^\gamma  (T-x-s)^\gamma\, ds\,dx
\notag\\
&=\int_0^T e^{\theta (T-x)}(T-x)^{2\alpha+\beta+\gamma+1} \int_0^{1} y^{2\alpha}(T-(T-x)y)^\gamma  (1-y)^\gamma\, dy\,dx
\notag\\
&= T^{2\alpha+\beta+2\gamma+1} e^{\theta T}
\notag\\*
&\quad\times
\int_0^T e^{-\theta x} \left(1-\frac{x}{T}\right)^{2\alpha+\beta+\gamma+1} \int_0^1 z^{2\alpha} (1-z)^\gamma \biggl(1-\left (1-\frac{x}{T}\right )z\biggr)^\gamma dz\,dx
\notag\\
&\eqqcolon T^{2\alpha+\beta+2\gamma+1} e^{\theta T} g(T).
\label{eq:h-via-g}
\end{align}
Hence, we need to study the asymptotic behaviour of $g(T)$
in three cases.

%
1. Let $\gamma\in(-1,-\frac12)$.
We write
\begin{equation}\label{eq:g11}
T^{2\gamma+1} g(T) =
\int_0^T x^{2\gamma+1} e^{-\theta x} \left(1-\frac{x}{T}\right)^{2\alpha+\beta+\gamma+1}
f_1\left (\frac{x}{T}\right )dx,
\end{equation}
where
\begin{align*}
f_1(y)&\coloneqq y^{-2\gamma-1} \int_0^1 z^{2\alpha} (1-z)^\gamma \left (1-\left (1-\frac{x}{T}\right )z\right)^\gamma dz
\\
&=\frac{\Gamma(2\alpha+1)\Gamma(\gamma+1)}{\Gamma(2\alpha+\gamma + 2)} y^{-2\gamma-1}\hyper\left (-\gamma,2\alpha+1;2\alpha+\gamma + 2; 1 - y\right).
\end{align*}
Here we have used the integral representation \eqref{eq:hypergeom-int} for the hypergeometric function.
Further, Proposition~\ref{prop:hyperlim1}\,\ref{prop-i} implies that
\begin{equation}\label{eq:g12}
f_1(y) \to \frac{\Gamma(\gamma+1)\Gamma(-2\gamma-1)}{\Gamma(-\gamma)}, \quad\text{as }y\downarrow0.
\end{equation}
If we define $f_1(y)$ at $y=0$ by the above convergence, we will obtain the continuous function on $[0,1]$. Therefore, $f_1$ is bounded on $[0,1]$, i.\,e.,
\begin{equation}\label{eq:g13}
f_1^*\coloneqq\sup_{y\in[0,1]} f_1(y) <\infty.
\end{equation}
In order to pass to the limit in \eqref{eq:g11}, we split the integral into two parts as follows:
$\int_0^T = \int_0^{T/2}+\int_{T/2}^T$. For the first part we can apply the Lebesgue dominated convergence theorem. Indeed, if $x\in[0,T/2]$, then
\[
\left(1-\frac{x}{T}\right)^{2\alpha+\beta+\gamma+1}
\le \max\left(1, 2^{-2\alpha-\beta-\gamma-1}\right)\eqqcolon c,
\]
Therefore
\[
x^{2\gamma+1} e^{-\theta x} \left(1-\frac{x}{T}\right)^{2\alpha+\beta+\gamma+1}
f_1\left (\frac{x}{T}\right ) \indicatorfun_{x\in[0,\frac T2]}
\le c f_1^* x^{2\gamma+1} e^{-\theta x}.
\]
Here the right-hand side is integrable on $(0,\infty)$:
\begin{equation}\label{eq:g14}
\int_0^\infty x^{2\gamma+1} e^{-\theta x}
= \frac{1}{\theta^{2\gamma+2}}\int_0^\infty e^{-y} y^{2\gamma+1}\,dy
= \frac{\Gamma(2\gamma+2)}{\theta^{2\gamma+2}}<\infty.
\end{equation}
Hence, letting $T\to\infty$ and taking into account \eqref{eq:g12} and \eqref{eq:g14}, we get
\begin{equation}\label{eq:g15}
\begin{split}
\MoveEqLeft
\int_0^{T/2} x^{2\gamma+1} e^{-\theta x} \left(1-\frac{x}{T}\right)^{2\alpha+\beta+\gamma+1}
f_1\left (\frac{x}{T}\right )dx
\\*
&\to f_1(0) \int_0^{\infty} x^{2\gamma+1} e^{-\theta x}\,dx
=\frac{\Gamma(\gamma+1)\Gamma(-2\gamma-1)\Gamma(2\gamma+2)}{\Gamma(-\gamma)\theta^{2\gamma+2}}.
\end{split}
\end{equation}
For $x\in[T/2,T]$, we use the bound $e^{-\theta x} \le e^{-\theta T/2}$ together with \eqref{eq:g13}.
We obtain
\begin{align}
\MoveEqLeft
\int_{T/2}^T x^{2\gamma+1} e^{-\theta x} \left(1-\frac{x}{T}\right)^{2\alpha+\beta+\gamma+1}
f_1\left (\frac{x}{T}\right )dx
\notag\\
&\le f_1^* e^{-\frac{\theta T}{2}}\int_{T/2}^T x^{2\gamma+1}  \left(1-\frac{x}{T}\right)^{2\alpha+\beta+\gamma+1}dx
\notag\\
&= f_1^* e^{-\frac{\theta T}{2}}T^{2\gamma+2}\int_{1/2}^1 y^{2\gamma+1}  \left(1-y\right)^{2\alpha+\beta+\gamma+1}dy
\notag\\
&\le f_1^* B(2\gamma+2,2\alpha+\beta+\gamma+2) e^{-\frac{\theta T}{2}}T^{2\gamma+2}
\to0,
\quad\text{as }T\to\infty,
\label{eq:g16}
\end{align}
where $B(a,b) = \int_0^1 y^{a-1}(1-y)^{b-1}dy$ is the beta function.
Combining \eqref{eq:g11} with \eqref{eq:g15}--\eqref{eq:g16} we arrive at
\[
T^{2\gamma+1} g(T) \to \frac{\sigma^2}{\theta^{2\gamma+2}},
\quad\text{as } T\to\infty.
\]
Taking into account the representation \eqref{eq:h-via-g}, we get the desired convergence \eqref{eq:h1}.

The cases $\gamma = -\frac12$ and $\gamma > -\frac12$ are considered similarly, so we shall omit some details.

2. In particular, for $\gamma = -\frac12$ we have
\begin{equation}\label{eq:g21}
\frac{g(T)}{\log T} = \int_0^T e^{-\theta x} \left(1-\frac{x}{T}\right)^{2\alpha+\beta+\frac12} \frac{\log \frac{T}{x}}{\log T}
f_2\left (\frac{x}{T}\right )dx,
\end{equation}
where
\[
f_2(y) = \frac{\Gamma(2\alpha+1)\Gamma(\frac12)}{\Gamma(2\alpha+\frac32)} \cdot \frac{\hyper\left (\frac12,2\alpha+1;2\alpha+\frac32; 1 - y\right)}{\log(1/y)}.
\]
In order to justify the passage to the limit in \eqref{eq:g21},
the logarithmic term $\frac{\log (T/x)}{\log T}$ can be bounded as follows: for all $x>0$ and for sufficiently large~$T$
\[
\frac{\log \frac{T}{x}}{\log T}
= \frac{\log T - \log x}{\log T}
= 1 - \frac{\log x}{\log T}
\le 1 + \frac{C_\delta x^{-\delta}}{\log T}
\le 1 + C_\delta x^{-\delta},
\]
where $\delta\in (0,1)$ can be chosen arbitrarily, $C_\delta$ is a positive constant.
Then arguing as above and taking into account the convergence $f_2(\frac{x}{T})\to1$ (by Proposition~\ref{prop:hyperlim1}\,\ref{prop-ii}),
we arrive at
\[
\frac{g(T)}{\log T} \to \int_0^\infty e^{-\theta x} \,dx = \frac{1}{\theta},
\quad\text{as } T\to\infty.
\]

3. Finally, for $\gamma > -\frac12$, we have
\[
g(T) =
\int_0^T e^{-\theta x} \left(1-\frac{x}{T}\right)^{2\alpha+\beta+\gamma+1}
f_3\left (\frac{x}{T}\right )dx,
\]
where
\begin{align*}
f_3(y)&\coloneqq \frac{\Gamma(2\alpha+1)\Gamma(\gamma+1)}{\Gamma(2\alpha+\gamma + 2)} \hyper\left (-\gamma,2\alpha+1;2\alpha+\gamma + 2; 1 - y\right)
\\
&\to \frac{\Gamma(2\alpha+1)\Gamma(2\gamma+1)}{\Gamma(2\alpha+\gamma + 2)}
\quad\text{as } y\downarrow0,
\end{align*}
by Proposition~\ref{prop:hyperlim1}\,\ref{prop-iii}.
Letting $T\to\infty$, we get
\[
g(T) \to  \frac{\Gamma(2\alpha+1)\Gamma(2\gamma+1)}{\Gamma(2\alpha+2\gamma + 2)}\int_0^\infty e^{-\theta x} \,dx
=  \frac{\Gamma(2\alpha+1)\Gamma(2\gamma+1)}{\Gamma(2\alpha+2\gamma + 2)\theta},
\]
where the passage to the limit is justified similarly to the proof of the first statement of the lemma.
\end{proof}
\end{appendix}

\bibliographystyle{amsplain}
\bibliography{mrs}

\end{document}